\documentclass[oneside,english]{amsart}
\usepackage[T1]{fontenc}
\usepackage[latin9]{inputenc}
\usepackage{amsmath,amsthm,amssymb,latexsym,epic,eepic,bbm,enumerate}
\usepackage{babel}
\usepackage[all]{xy}
\usepackage[parfill]{parskip}

\usepackage[active]{srcltx}
\usepackage[notref,notcite]{showkeys}

\newtheorem{theorem}{Theorem}%[section]
\newtheorem{lemma}[theorem]{Lemma}%[section]
\newtheorem{corollary}[theorem]{Corollary}%[section]
\numberwithin{equation}{section}

\theoremstyle{plain}% default

\theoremstyle{definition}

\theoremstyle{remark}

\newcommand{\K}{\Bbbk}

\newcommand{\eff}{\ensuremath{\mathrm{eff.dim}_{\K}}}

\newcommand{\e}{\varepsilon}

\newcommand{\set}[1]{\ensuremath{\{#1\}}}

\begin{document}
\title{Effective representations of Path Semigroups}
\author{Love Forsberg}

\begin{abstract}
We give a formula which determines the minimal effective dimensions of path semigroups and truncated path 
semigroups over an uncountable field of characteristic zero. 
\end{abstract}\maketitle

\section{Introduction and preliminaries}\label{s0}

Let $S$ be a semigroup and $\Bbbk$ a fixed field. In the paper \cite{MS} Mazorchuk and Stienberg addressed the 
question of dermining the so-called {\em minimal effective dimension} $\mathrm{eff.dim}_{\Bbbk}(S)$ of $S$ over 
$\Bbbk$, that is the minimal $m$ (a positive integer or infinity) for which there is an injective homomorphism 
from $S$ to the semigroup  $\mathrm{Mat}_{m\times m}(\Bbbk)$ of all $m\times m$ matrices with coefficients in 
$\Bbbk$. If $S$ is finite, it is clear that $\mathrm{eff.dim}_{\Bbbk}(S)<\infty$, more precisely,
$\mathrm{eff.dim}_{\Bbbk}(S)\leq|S|+1$, which is the dimension of the regular representation of the semigroup
$S^1$ obtained from $S$ by formally adjoining an identity element $1$. %In fact, \cite[Theorem~4]{MS} shows even that $\mathrm{eff.dim}_{\Bbbk}(S)<|S|$ and that this bound is sharp. 
An {\em effective representation} of a semigroup $S$ with $\mathrm{eff.dim}_{\Bbbk}(S)=m$ is an injective homomorphism $S\to \mathrm{Mat}_{m\times m}(\Bbbk)$.

One of the examples considered in \cite{MS} was that of {\em truncated path semigroups} which we now define.
Let $Q=(Q_0,Q_1,h,t)$ be a {\em quiver}, where $Q_0$ is the set of vertices, $Q_1$ is the set of arrows,
$h:Q_1\to Q_0$ is the function assigning to each arrow its {\em head} and 
$t:Q_1\to Q_0$ is the function assigning to each arrow its {\em tail}. Denote by $\mathcal{P}$ the set
of all oriented paths in $Q$ (including the trivial path $\varepsilon_x$ at each vertex $x\in Q_0$ and the
zero path $\mathtt{z}$). Then $\mathcal{P}$ has the natural structure of a semigroup under the usual 
concatenation of oriented paths (in case two paths cannot be concatenated, their product is postulated to be 
the path $\mathtt{z}$ and the latter is the zero element of $\mathcal{P}$). The semigroup 
$\mathcal{P}$ is called the {\em path semigroup} of $Q$. The semigroup $\mathcal{P}$ is finite if and only if the
quiver $Q$ is finite and does not have oriented cycles. We write $\mathcal{P}^*$ for the set
$\mathcal{P}\setminus \{\mathtt{z}\}$ of non-zero paths.

There is a unique function $\mathfrak{l}:\mathcal{P}^*\to\{0,1,2,\dots\}$, called the {\em path length}, having 
the properties that the length of each arrow is $1$ and $\mathfrak{l}(pq)=\mathfrak{l}(p)+\mathfrak{l}(q)$ 
whenever $p,q,pq\in \mathcal{P}^*$ (note that $\mathfrak{l}(\varepsilon_x)=0$ for each $x\in Q_0$). 
Elements of $Q_1$ thus can be identified with all paths of length $1$. Let $J\subset \mathcal{P}$ be the 
two-sided ideal of $\mathcal{P}$ generated by $Q_1$. For every $N\in\{1,2,3,\dots\}$ we define
the \emph{truncated path semigroup} as $\mathcal{P}_N:=\mathcal{P}/J^N$. Note that the semigroup $\mathcal{P}_N$ is finite whenever $Q$ is. In \cite[Subsection~8.1]{MS} one finds a formula for the effective dimension of $\mathcal{P}_N$ in the case when every vertex in $Q$ appears in some oriented cycle (or loop). The aim of this paper is give a formula for the effective dimension of $\mathcal{P}_N$ for any $Q$. 

From now on we assume that $Q$ is finite and set $n=|Q_0|$. Consider the path algebra $\Bbbk[Q]$ of $Q$ 
which is the $\mathtt{z}$-reduced semigroup algebra of $\mathcal{P}$ over $\Bbbk$. The algebra $\Bbbk[Q]$ is 
unital with unit element $1=\sum_{x\in Q_0}\varepsilon_x$ where $\varepsilon_x$ are pairwise orthogonal
idempotents. This implies that any $\Bbbk[Q]$-module $V$ splits as a direct sum of vector spaces 
\[V=\bigoplus_{x\in Q_0}V_x,\]
where $V_x=\varepsilon_x V$. Given a $\Bbbk[Q]$-module $V$, we set $D_x=\dim(V_x)$. Each arrow
$\alpha:x\to y$ acts as zero on all $V_z$ such that $z\neq x$ and hence is uniquely determined by the 
induced linear map from $V_x$ to $V_y$. Hence we can make the following convention: A matrix representation of 
$\mathcal{P}$ or $\mathcal{P}_N$ is an assignment to each arrow $\alpha\in Q_1$ a $D_y\times D_x$-matrix 
with coefficients in $\Bbbk$ representing the  action of $\alpha$ in fixed bases of $V_x$ and $V_y$. For more details on representations of quivers we refer the reader to \cite{GR}.

Note that $\mathcal{P}_N$-modules are exactly $\mathcal{P}$-modules annihilated by $J^N$. We will usually denote 
$\mathcal{P}$- or $\mathcal{P}_N$-modules by  $V$ and the corresponding representation by $R$. 

\section{Path semigroups}\label{s1}

In \cite[Section~8]{MS} one finds formulae for effective dimension of path semigroups over $\Bbbk$ in case
of acyclic $Q$ and algebraically closed $\Bbbk$. In this section we determine the effective dimension of 
path semigroups for all finite quivers at the expense of assuming $\Bbbk$ to be a field 
containing an infinite purely transcendental extension of its prime subfield (for example,
$\mathbb{R}\subset\Bbbk$). We denote by $\mathbb{N}$ the set of positive integers and by 
$\mathbb{N}_0$ the set of non-negative integers.

For $x\in Q_0$ let $\mathcal{P}_x$ denote the set of all paths in $\mathcal{P}^*$ which start and terminate
at $x$. Then $\mathcal{P}_x$ is a subsemigroup of $\mathcal{P}$, in fact, $\mathcal{P}_x$ is a monoid
with identity $\varepsilon_x$. Denote by $A$ the set of all vertices $x\in Q_0$ for which
$\mathcal{P}_x$ is not commutative and set $B:=Q_0\setminus A$.

\begin{lemma}\label{lem101}
Let $x\in Q_0$.
\begin{enumerate}[$($i$)$]
\item\label{lem101.1} The monoid $\mathcal{P}_x$ has a unique irreducible generating system (which we 
denote by $M_x$).
\item\label{lem101.2} The monoid $\mathcal{P}_x$ is free over $M_x$.
\item\label{lem101.3} The monoid $\mathcal{P}_x$ is commutative if and only if $|M_x|\leq 1$. 
\end{enumerate}
\end{lemma}

\begin{proof}
Define $N_i$ and $\tilde{N}_i$ for $i\in\mathbb{N}$ recursively as follows: 
\begin{itemize}
\item $N_1$ is the set of all paths of length $1$ in $\mathcal{P}_x$;
\item $\tilde{N}_1$ is the subsemigroup of $\mathcal{P}_x$ generated by $N_1$;
\item $N_2$ is the set of all paths of length $2$ in $\mathcal{P}_x\setminus \tilde{N}_1$;
\item $\tilde{N}_2$ is the subsemigroup of $\mathcal{P}_x$ generated by $N_1\cup N_2$;
\item $N_3$ is the set of all paths of length $3$ in $\mathcal{P}_x\setminus \tilde{N}_2$;
\item $\tilde{N}_3$ is the subsemigroup of $\mathcal{P}_x$ generated by $N_1\cup N_2\cup N_3$;
\item and so on.
\end{itemize}
From this definition it is clear that the set $M_x=N_1\cup N_2\cup\dots$ is a generating system of 
$\mathcal{P}_x$ (as a monoid) and that it is included in every generating system of 
$\mathcal{P}_x$ (as a monoid). Claim~\eqref{lem101.1} follows.

Assume that $\mathcal{P}_x$ is not free over $M_x$. Then there exist $a_1,a_2,\dots,a_k,b_1,b_2,\dots,b_l\in M_x$ 
such that $a_1a_2\cdots a_k=b_1b_2\cdots b_l$. This can be chosen such that $(k,l)$ is minimal possible
with respect to the lexicographic order (note that, obviously, both $k,l>0$). If $\mathfrak{l}(a_k)<\mathfrak{l}(b_l)$, 
then $b_l=ta_k$ for some $t\in \mathcal{P}_x$ which contradicts $b_l\in M_x$. Therefore this case is not possible. 
Similarly $\mathfrak{l}(a_k)>\mathfrak{l}(b_l)$ is not possible. This means that $\mathfrak{l}(a_k)=\mathfrak{l}(b_l)$ 
and hence $a_k=b_l$. Therefore $a_1a_2\dots a_{k-1}=b_1b_2\dots b_{l-1}$ which contradicts minimality of $(k,l)$.
This proves claim~\eqref{lem101.2} and claim~\eqref{lem101.3} follows directly from claim~\eqref{lem101.2}. 
\end{proof}

A generator of $\mathcal{P}_x$ will be called a {\em minimal oriented cycle} starting at $x$

\begin{lemma}\label{lem102}
Let $V$ be an effective $S$-module and $x\in A$. Then $D_x\geq 2$ and 
\[\eff(S)\geq 2|A|+|B|=|A|+n.\]
\end{lemma}

\begin{proof}
It is clear that $D_x\geq 1$ for all $x\in Q_0$ (for otherwise the actions of $\e_x$ and $\mathtt{z}$ would coincide).
Assume $x\in A$ and $D_x=1$. Then $\mathcal{P}_x$ acts effectively acts on the $1$-dimensional vector space $V_x$.
However, the semigroup of linear endomorphisms of $V_x$ is commutative (as $V_x$ is one dimensional), while
$\mathcal{P}_x$ is not (as $x\in A$), a contradiction. This implies
that $D_x>1$ for $x\in A$, that is $D_x\geq 2$. As $|A|+|B|=|Q_0|=n$, the claim of the lemma follows.
\end{proof}

To prove that the bound given by Lemma~\ref{lem102} is sharp, we will need the following construction:
For a fixed positive integer $k$ consider the alphabet $A=\{a_1,a_2,\dots,a_k\}$ and the free monoid
$A^*$ of all finite words over $A$ with respect to concatenation of words. Let $\Bbbk_k$ be the purely 
transcendental extension of its prime subfield $\mathbb{K}$ with basis 
$\mathbf{B}:=\{\tau_{i},\eta_{i},\zeta_{i}\,|\,i=1,2,\dots,k\}$.

\begin{lemma}\label{lem105}
There is a unique representation $R:A^*\to \mathrm{Mat}_{2\times 2}(\Bbbk_k)$ such that 
\begin{displaymath}
R(a_i)= \left(\begin{array}{cc}\tau_{i}&\eta_{i}\\0&\zeta_{i}\end{array}\right),
\end{displaymath}
moreover, the map $R$ is injective.
\end{lemma}

\begin{proof}
Existence and uniqueness of $R$ follows from the fact that  $A^*$ is free over $A$. For 
$a_{i_1}a_{i_2}\dots a_{i_l}\in A^*$ the coefficient in the first row and second column
of the matrix $R(a_{i_1}a_{i_2}\dots a_{i_l})$ equals
\begin{displaymath}
\sum_{i=1}^l \tau_{1}\tau_{2}\cdots \tau_{i-1}\eta_i\zeta_{i+1}\zeta_{i+2}\cdots\zeta_{l}.
\end{displaymath}
This uniquely determines the sequence $i_1,i_2,\dots,i_l$ and the claim about injectivity follows.
\end{proof}

For a fixed $Q$ let $\Bbbk_Q$ be the purely transcendental extension of its prime subfield $\mathbb{K}$ with basis 
$\mathbf{B}:=\{\tau_{\alpha},\eta_{\alpha},\zeta_{\alpha}\,|\,\alpha\in Q_1\}$. 
For  $\alpha\in Q_1$ set $\mathbf{B}_{\alpha}:=\{\tau_{\alpha},\eta_{\alpha},\zeta_{\alpha}\}$.
For $x,y\in Q_0$ write $x\sim y$ if $x=y$ or there is an  oriented path from $x$ to $y$ 
as well as an oriented path from $y$ to $x$.

\begin{lemma}\label{lem107}
Let $x,y\in Q_0$ be such that $x\sim y$. Then $x\in A$ if and only if $y\in A$.
\end{lemma}

\begin{proof}
As $x\sim y$, there exist paths $\omega_{xy}:x\to y$ and $\omega_{yx}:y\to x$. Assume $x\in A$. Let $\omega_1$ and $\omega_2$ be two different minimal oriented cycles in $\mathcal{P}_x$. Then $\omega_{xy}\omega_1\omega_{yx}$ and $\omega_{xy}\omega_2\omega_{yx}$ are  two noncommuting elements in  
$\mathcal{P}_x$, proving $y\in A$. Claim now follows by symmetry.
\end{proof}

The following is our first main result.

\begin{theorem}\label{thm21}
Let $Q$ be a finite quiver and $\mathcal{P}$ the corresponding path semigroup. Then
\[\mathrm{eff.dim}_{\Bbbk_Q}(\mathcal{P})=|A|+n.\]
\end{theorem}

\begin{proof}
We only need to show that the bound given by  Lemma~\ref{lem102} is sharp.
To do this we construct an effective matrix representation $V$ of $\mathcal{P}$ as follows: set 
\[D_x=\dim(V_x)=\begin{cases}2,&x\in A;\\ 1,& x\in B;\end{cases}\]
with a fixed basis in each $V_x$. To each $\alpha\in Q_1$ we assign a $\Bbbk_Q$-matrix with
$D_{h(\alpha)}$ rows and $D_{t(\alpha)}$ columns by the following rule:
\begin{itemize}
\item If $h(\alpha_i),t(\alpha_i)\in A$, then we assign to $\alpha$ the matrix
$\left(\begin{array}{cc}\tau_{\alpha}&\eta_{\alpha}\\0&\zeta_{\alpha}\end{array}\right)$. 
\item If $h(\alpha_i)\in A$ and $t(\alpha_i)\in B$,
then we assign to $\alpha$ the matrix $(\tau_{\alpha}\,\,\,\zeta_{\alpha})$.
\item If $h(\alpha_i)\in B$ and $t(\alpha_i)\in A$,
then we assign to $\alpha$ the matrix
$\left(\begin{array}{c}\tau_{\alpha}\\\zeta_{\alpha}\end{array}\right)$. 
\item If $h(\alpha_i),t(\alpha_i)\in B$,
then we assign to $\alpha$ the matrix $(\tau_{\alpha})$.
\end{itemize}
Finally, to each $\varepsilon_x$ we assign the identity matrix of size $D_x$ and to each path of length more than $1$ the corresponding product of the matrices assigned to arrows which this path consists of. It is obvious that this
gives a well-defined representation of $\mathcal{P}$. It remains to show that this representation sends different
elements of $\mathcal{P}$ to different linear operators.

Let $x,y\in Q_0$ and $\omega$ be an oriented paths from $x$ to $y$. Directly from the above construction it follows
that each coefficient of the matrix representing $\omega$ is a homogeneous polynomial in elements from $\mathbf{B}$.
If this coefficient is nonzero (which is the case for all diagonal entries and all entries above the diagonal,
in case the latter exist), this polynomial has degree $\mathfrak{l}(\omega)$ and depends on at least one
element from $\{\tau_{\alpha},\eta_{\alpha},\zeta_{\alpha}\}$ for each arrow $\alpha$ in $\omega$.

Let $x,y\in Q_0$ and $\omega,\omega'$ be two paths from $x$ to $y$. We have to show that $\omega$ and $\eta$ are
represented by different linear operators. From the previous paragraph it follows that this is clear in the case
when $\omega$ and $\omega'$ have different lengths and in the case when one of these paths contains an arrow which 
is not contained in the other path. 

Assume that there exists $x,y\in Q_0$ and $\omega,\omega'$ two different paths from $x$ to $y$ such that
$R(\omega)=R(\omega')$. Without loss of generality we may assume that the pair
$(\mathfrak{l}(x),\mathfrak{l}(y))$ is minimal with respect to the lexicographic order.

Write $\omega$ in the form $\omega_1\beta_1\omega_2\beta_2\cdots\omega_{k-1}\beta_{k-1}\omega_k$ where 
$\omega_i$ are (possibly trivial) paths inside an equivalence class of the relation $\sim$ and
$\beta_i$ are arrow between equivalence classes. From the above it then follows that
$\omega'$ can similarly be written as $\omega'_1\beta_1\omega'_2\beta_2\cdots\omega'_{k-1}\beta_{k-1}\omega'_k$.

Assume $\omega_1$ is a trivial path. Then $\omega$ has no arrow starting from the $\sim$-equivalence class
of $h(\beta_1)$. From the above we get that $\omega'$ has no arrow starting from the $\sim$-equivalence class
of $h(\beta_1)$ and hence $\omega'_1$ is a trivial path as well.  We claim that this implies
\begin{equation}\label{eq1}
R(\omega_2\beta_2\cdots\omega_{k-1}\beta_{k-1}\omega_k)=R(\omega'_2\beta_2\cdots\omega'_{k-1}\beta_{k-1}\omega'_k)
\end{equation}
which would then contradict the minimality of $(\mathfrak{l}(x),\mathfrak{l}(y))$. To prove \eqref{eq1},
the only non-trivial case to consider is when $R(\beta_1)$ is not injective, that is 
$t(\beta_1)\in A$ and $h(\beta_1)\in B$. Assume 
\begin{displaymath}
R(\omega_2\beta_2\cdots\omega_{k-1}\beta_{k-1}\omega_k)=
\left(\begin{array}{cc}a&b\\0&c\end{array}\right)\neq
\left(\begin{array}{cc}a'&b'\\0&c'\end{array}\right)=
R(\omega'_2\beta_2\cdots\omega'_{k-1}\beta_{k-1}\omega'_k).
\end{displaymath}
Then none of $a,b,c,a',b',c'$ depends on $\tau_{\beta_1}$ or $\zeta_{\beta_1}$ and hence we have
\begin{multline*}
R(\omega)=
\left(\begin{array}{cc}\tau_{\beta_1}&\zeta_{\beta_1}\end{array}\right)
\left(\begin{array}{cc}a&b\\0&c\end{array}\right)=
\left(\begin{array}{cc}\tau_{\beta_1}a&\tau_{\beta_1}b+\zeta_{\beta_1}c\end{array}\right)\neq\\\neq
\left(\begin{array}{cc}\tau_{\beta_1}a'&\tau_{\beta_1}b'+\zeta_{\beta_1}c'\end{array}\right)=
\left(\begin{array}{cc}\tau_{\beta_1}&\zeta_{\beta_1}\end{array}\right)
\left(\begin{array}{cc}a'&b'\\0&c'\end{array}\right)=R(\omega'),
\end{multline*}
a contradiction.

Therefore $\omega_1$ is non-trivial and thus $R(\omega_1)$ is invertible by construction and Lemma~\ref{lem107} 
as both the starting point and the ending point of $\omega_1$  belong to the same $\sim$-equivalence class.
Multiplying with $R(\omega_1)^{-1}$ we get
\begin{displaymath}
R(\beta_1\omega_2\beta_2\cdots\omega_{k-1}\beta_{k-1}\omega_k)=
R(\omega_1)^{-1}R(\omega'_1)R(\beta_1\omega'_2\beta_2\cdots\omega'_{k-1}\beta_{k-1}\omega'_k).
\end{displaymath}
Note that the left hand side does not depend on elements in $\mathbf{B}_{\alpha}$ for $\alpha$ occurring in 
$\omega_1$. Hence the right hand side does not depend on these elements either which forces the injective linear map
$R(\omega_1)^{-1}R(\omega'_1)$ to be the identity linear map as the image of the linear map
$R(\beta_1\omega'_2\beta_2\cdots\omega'_{k-1}\beta_{k-1}\omega'_k)$ is nonzero by construction. Therefore 
in this case we have the equality
$R(\omega_1)=R(\omega'_1)$. If $\beta_1\omega_2\beta_2\cdots\omega_{k-1}\beta_{k-1}\omega_k$
or $\beta_1\omega'_2\beta_2\cdots\omega'_{k-1}\beta_{k-1}\omega'_k$ is non-trivial, the above gives
\begin{displaymath}
R(\beta_1\omega_2\beta_2\cdots\omega_{k-1}\beta_{k-1}\omega_k)=
R(\beta_1\omega'_2\beta_2\cdots\omega'_{k-1}\beta_{k-1}\omega'_k)
\end{displaymath}
which contradicts minimality of $(\mathfrak{l}(x),\mathfrak{l}(y))$. Hence $\omega=\omega_1$
and $\omega'=\omega'_1$.

If $x\in A$, then $R(\omega)=R(\omega')$ implies $\omega=\omega'$ by Lemma~\ref{lem105}, a contradiction.
Therefore $x,y\in B$. In this case there is a unique minimal oriented cycle $q$ from $x$ to $x$ 
($q$ may be a trivial path) and hence
a unique path $p$ of minimal length from $x$ to $y$ (for otherwise, composing two different such minimal 
paths from $x$ to $y$ with a minimal path from $y$ to $x$ we would get
two minimal oriented cycles from $x$ to $x$). Any path from $x$ to $y$ has thus the form $pq^l$ for some
positive integer  $l$. In particular, two paths of the same length from $x$ to $y$ must coincide,
which contradicts our choice of $\omega$ and $\omega'$. This final contradiction completes the proof of the
theorem.
\end{proof}

\section{Truncated path semigroups}\label{s3}

As truncated path semigroups are obtained by adding some relations to usual path semigroups, it is reasonable
to expect that the effective dimension increases, e.g.  compare the statements of Theorem~\ref{thm21}
above with the results of \cite[Subsection~8.2]{MS}.

Let $\Bbbk$ be any field, $N\in\mathbb{N}$ and $V$ a representation of $\mathcal{P}_N$.
For every $k\in\mathbb{N}_0$ let $W^{(k)}=\mathrm{span}\set{\omega V\,|\,\omega\in\mathcal{P},\,\,l(\omega)=k}$. 
By convention, $\omega=\mathtt{z}$ 
when $\mathfrak{l}(\omega)\geq N$, which gives $W^{(N)}=0$. Thus we get the chain of subspaces
\[V=W^{(0)}\supset W^{(1)}\supset\cdots\supset W^{(N-1)}\supset W^{(N)}=0.\]
For every $x\in Q_0$ set $W_x^{(k)}:=V_x\cap W^{(k)}$ and choose \emph{some}  
$V_x^{(k)}\subset W^{(k)}_x$ such that $W_x^{(k)}=V_x^{(k)}\oplus W_x^{(k+1)}$.
Set $\displaystyle V^{(k)}:=\bigoplus_{x\in Q_0}V_x^{(k)}$. This gives 
the vector space decompositions
\[V=\bigoplus_{i=0}^{N-1}V^{(i)}=\bigoplus_{x\in Q_0}V_x=\bigoplus_{\substack{0\leq i\leq N-1\\x\in Q_0}}V_x^{(i)}.\]
In any module $V$, let $D_x^{(i)}:=\dim(V_x^{(i)})$, which gives $D_x=\sum_{i=0}^{N-1}D_x^{(i)}$.
From the definition of $W^{(i)}$ for any $\omega\in\mathcal{P}$ we have 
$\omega W^{(i)}\subset W_{h(\omega)}^{(i+\mathfrak{l}(\omega))}$.

\begin{lemma}\label{lem202}
Let $x\in Q_0$ be such that there are paths $\omega_l$, $\omega_r$ and $0\leq k<N$ such that 
$\mathfrak{l}(\omega_l)=k$, $\mathfrak{l}(\omega_r)=N-1-k$ and $h(\omega_l)=t(\omega_r)=x$. 
Then $D_x^{(k)}\geq 1$ for every effective $\mathcal{P}$-module $V$.
\end{lemma}

\begin{proof} 
Assume $D_x^{(k)}=0$, that is $V_x^{(k)}=0$, and let $y=t(\omega_l)$ and $z=h(\omega_r)$. Then 
\[\omega_l (V)=\omega_l(V_y)=\omega_l(W_y^{(0)})\subset W_x^{(l(\omega_l))}=W_x^{(k)}=V_x^{(k)}\oplus W_x^{(k+1)}=W_x^{(k+1)}\mathrm{\ and}\]
\[\omega_r (W_x^{(k+1)})\subset W_z^{(k+1+l(\omega_r))}=W_z^{(k+1+N-1-k)}=W_z^{(N)}=0.\]
Thus $\omega_r\omega_l(V)=0$ and $\omega_r\omega_l$ acts as $\mathtt{z}$ on $V$ contradicting effectiveness.
\end{proof}

For $x\in Q_0$ define 
\begin{multline*}
K(x):=\big\{k\in \{0,\cdots,N-1\}\,|\,
\text{ there are paths } \omega_l,\omega_r\text{ such that }\\
\mathfrak{l}(\omega_l)=k,\mathfrak{l}(\omega_r)=N-1-k\text{ and }h(\omega_l)=t(\omega_r)=x\big\}. 
\end{multline*}
Set $B:=\{x\in Q_0\,|\, K(x)=\varnothing\}$ and $A:=Q_0\setminus B$. For $x\in A$ set
\[\underline{k}_x:=\min(K(x))\quad\text{ and } \quad \overline{k}_x:=\max(K(x)).\]
For $x\in Q_0$ define
\[l_x^-:=\sup\{\mathfrak{l}(\omega)\,|\,\omega\in\mathcal{P}\text{ and } h(\omega)=x\} \text{ and }
l_x^+:=\sup\{\mathfrak{l}(\omega)\,|\,\omega\in\mathcal{P}\text{ and } t(\omega)=x\}.\]
We are now ready to state our second main result.

\begin{theorem}\label{thm205}
Define $d_x:=\min\big\{l_x^-+1,l_x^++1,N,\max\{l_x^-+l_x^++2-N,1\}\big\}$. 
\begin{enumerate}[$($i$)$]
\item\label{thm205.1} For every effective $\mathcal{P}_N$-module $V$ over any field $\Bbbk$ we have $D_x\geq d_x$.
\item\label{thm205.2} If $\Bbbk$ has characteristic zero or is uncountable, then $D_x=d_x$ for some effective $\mathcal{P}_N$-module (over $\K)$ and
$\mathrm{eff.dim}_{\Bbbk}(\mathcal{P}_N)=\sum_{x\in Q_0}d_x$. 
\end{enumerate}
\end{theorem}

\begin{proof}
First we {\bf prove claim~\eqref{thm205.1}}. Let $x\in Q_0$. Then $x\in A$ or $x\in B$.
In any case, $D_x\geq |K(x)|$ by Lemma~\ref{lem202}. 

{\bf Assume first that $x\in A$.} Then $K(x)\neq\varnothing$ and 
it suffices to show that $|K(x)|\geq d_x$. As $K(x)\neq\varnothing$, there is 
some path of length $N-1$ passing through $x$, which means that $l_x^-+l_x^+\geq N-1$, in particular,
$l_x^-+l_x^+-N+2\geq 1$ and thus $\max\{l_x^-+l_x^++2-N,1\}=l_x^-+l_x^+-N+2$. 

Pick some paths $\omega_-,\omega_+$ such that $h(\omega_-)=t(\omega_+)=x$ and $l(\omega_\pm)=\min(l_x^\pm,N-1)$. 
Let $\omega_{\min(l_x^-,N-1)}$ be a path of length $N-1$ that starts with $\omega_-$ and continues into $\omega_+$ 
(if needed). From Lemma~\ref{lem202} we get $\min(l_x^-,N-1)\in K(x)$. Now we repeat recursively the following 
procedure as long as possible: Change $\omega_{k}$ to $\omega_{k-1}$ by removing the tail arrow and adding a new head 
arrow from $\omega_+$. On each step of this procedure we get a new $\omega_{k-1}$ with $k-1\in K(x)$. This procedure 
can stop for two reasons:
\begin{itemize}
\item There are no more arrows from $\omega_-$ to remove.
\item There are no more arrows from $\omega_+$ to add.
\end{itemize}
The first case (there are no more arrows from $\omega_-$ to remove) can only happen if the latest  
$k-1$ created is equal to $0$. In this case $K(x)\supset\{0,1,\cdots,\min(l_x^-,N-1)\}$ and hence 
$|K(x)|\geq\min(l_x^-+1,N)$ and we are done. 

We split the second case (there are no more arrows from $\omega_+$ to add) into two subcases.
The first subcase is that $\omega_{\min(l_x^-,N-1)}=\omega_-$, that is $l_x^-\geq N-1$. In this subcase 
we have $K(x)\supset\{N-1,N-2,\cdots,N-1-\min(l_x^+,N-1)\}$ which implies that  $|K(x)|\geq\min(l_x^++1,N)$
and we are done. 

The second subcase is when $\omega_{\min(l_x^-,N-1)}\neq\omega_-$. In this subcase we have  $l_x^-<N-1$ and 
\[K(x)\supset T:=\{l_x^-,l_x^--1,\cdots,N-1-\min(l_x^+,N-1)\}.\] 
Hence $|K(x)|\geq |T|=(l_x^-+\min(l_x^+,N-1)+2-N)$. If $\min(l_x^+,N-1)=l_x^+$, this gives
$|K(x)|\geq l_x^-+l_x^++2-N$ and we are done. If $\min(l_x^+,N-1)=N-1$, this gives
$|K(x)|\geq l_x^-+1$ and we are done. This completes verification of $D_x\geq|K(x)|\geq d_x$ for $x\in A$.

{\bf Assume now that $x\in B$.} In this case $l_x^-+l_x^++2-N\leq 0$ and $d_x=1$. The fact that $D_x\geq 1$
is clear as $\varepsilon_x$ acts as the identity on $V_x$ and this should be different from the action of 
$\mathtt{z}$ which acts as zero. This completes the proof of claim~\eqref{thm205.1} and implies
\[\mathrm{eff.dim}_{\Bbbk}(\mathcal{P}_N)\geq\sum_{x\in Q_0}d_x.\]
		
To {\bf prove claim~\eqref{thm205.2}} we assume that $\Bbbk$ has characteristic zero or is uncountable. We have to construct an
effective representation $V$ such that $D_x=d_x$ for every $x\in Q_0$. 
To do this we define the following:
\begin{itemize}
\item for $x\in A$ and $k\in K(x)$ let $V_x^{(k)}$ be the one-dimensional vector space with basis $\{v_x^{(k)}\}$;
\item for $x\in A$ and $k\not\in K(x)$ let $V_x^{(k)}$ be the zero vector space;
\item for $x\in B$ let $V_x$ be the one-dimensional vector space with basis $\{v_x\}$.
\end{itemize}
Set
\[V:=\big(\bigoplus_{\substack{0\leq i\leq N-1\\x\in A}}V_x^{(i)}\big)\oplus\big(\bigoplus_{x\in B}V_x\big).\]
Fix an injective map $(\alpha,k)\mapsto p_{\alpha,k}$ from the set of all pairs $(\alpha,k)$ where
$\alpha\in Q_1$  and $0\leq k\leq N$ to the set of positive integer prime numbers if $\Bbbk$ has charachteristic 0. In case $\Bbbk$ is uncountable we choose the codomain as a basis of a purely transcendental extension over its prime subfield by sufficiently many base elements. 
Define the action of $\mathcal{P}_N$ on $V$ as follows:
\begin{itemize}
\item the zero element of $\mathcal{P}_N$ acts as zero;
\item $\e_x$ acts as the identity on $V_x$ and as zero on $V_y$, $y\neq x$;
\item for every arrow  $\alpha:x\to y$ with $x,y\in A$ we have $\alpha:v_x^{(N-1)}\to 0$ and for
each $k\in K(x)$ we have 
$\alpha:v_x^{(k)}\to p_{\alpha,k}v_y^{(j)}$, where 
\[j=\min\{i\in \{k+1,k+2,\dots,N-1\}\,|\,V_y^{(i)}\neq 0\};\]
\item for every arrow $\alpha:x\to y$ with $x\in A$ and $y\in B$ and for each $k\in K(x)$
we have $\alpha:v_x^{(k)}\to p_{\alpha,k}v_y$;
\item for every arrow $\alpha:x\to y$ with $x\in B$ and $y\in A$ we have $\alpha:v_x\to p_{\alpha,0}v_y^{(\underline{k}_y)}$;
\item for every arrow $\alpha:x\to y$ with $x,y\in B$ we have $\alpha:v_x\to p_{\alpha,0}v_y$;
\item actions of paths of length greater than one are defined using composition of maps.			
\end{itemize}
Assume that $x,y\in A$ and $k\in K(x)$. Let $\omega_-$ and $\omega_+$ be two paths such that 
$\mathfrak{l}(\omega_-)=k$, $\mathfrak{l}(\omega_+)=N-1-k$ and $h(\omega_-)=t(\omega_+)=x$. 
Assume further that there is an arrow $\alpha$ from $x$ to $y$.
If $N-1-k\leq l_y^++1$, then without loss of generality we may assume that 
$\alpha$ is the first arrow in $\omega_+$. In this case we directly get $k+1\in K(y)$.
If $l_y^++1<N-1-k$, then any $k'\in K(y)$ satisfies $N-1-k'\leq l_y^+< N-1-k$ which implies $k'>k$.
Since $K(y)$ is not empty (as $y\in A$), we get that the set 
$\{i\in \{k+1,k+2,\dots,N-1\}|V_y^{(i)}\neq 0\}$ is non-empty. Therefore the above definitions make sense.

The only non-trivial relation to check is the fact that any path $\omega$ with $\mathfrak{l}(\omega)\geq N$ 
acts as zero. From the definition of $B$ it follows that each arrow in $\omega$ is an arrow between two vartices 
in $A$. From the definition of the action we then see that 
\[\omega(V_{t(\omega)}^{(i)})\subset V^{(i+\mathfrak{l}(\omega))}_{h(\omega)}.\]
This implies $\omega(V)\subset 0$ and thus $V$ is a $\mathcal{P}_N$-module.

It remains to show that our module is effective. For this we need to show that paths of length at most $N-1$
act in a non-zero way and pairwise differently. A path $\omega$ is said to be {\em maximal} if there 
is no arrow $\alpha$ such that $\alpha\omega$ or $\omega\alpha$ is nonzero. Note that if a path 
$\omega$ acts in a nonzero way,  then $h(\omega)$ can be recovered as the unique $y$ such that 
$\omega(V)\subset V_y$. Moreover, $t(\omega)$ can be recovered as the unique $x$ such that $\omega(V_x)\neq 0$.
Thus if two different paths $\omega_1$ and $\omega_2$ act equally and in a nonzero way, then they share the same 
head and the same tail. Furtheremore, the action of each maximal path $\omega_l\omega_1\omega_r$ coincides with
the action of $\omega_l\omega_2\omega_r$. Thus it suffices to show that all maximal paths act nonzero and differently. 

To simplify notation let
\[\hat{v}_x:=\begin{cases}v_x^{(\underline{k}_x)}, &x\in A;\\v_x,& x\in B; \end{cases}\qquad\qquad
\check{v}_y:=\begin{cases}v_y^{(\overline{k}_x)}, &y\in A;\\v_y,& y\in B. \end{cases}\]
Let $\omega=\alpha_{N-1}\alpha_{N-2}\cdots\alpha_2\alpha_{1}$ 
be a path of length $N-1$ and set $x_i:=h(\alpha_i)=t(\alpha_{i+1})$ with $x_0:=t(\alpha_{1})$.
Then from Lemma~\ref{lem202} and our construction we get that $V_{x_i}^{(i)}$ is nonzero for all $i$ and $\omega(\hat{v}_x)=p_{\alpha_1,0}p_{\alpha_2,1}\cdots p_{\alpha_{N-1},N-2}\check{v}_y$. Injectivity of the
map $(\alpha,k)\mapsto p_{\alpha,k}$ guarantees that the coefficient at $\check{v}_y$ uniquely determines
the sequence $(\alpha_1,0),(\alpha_2,1),\dots,(\alpha_{N-1},N-2)$ which uniquely deretmines $\omega$. 

Finally, assume that $\omega=\alpha_k\alpha_{k-1}\cdots\alpha_2\alpha_1$ for some $k<N-1$. 
Set $w_0=\hat{v}_x$ and $w_i=\alpha_{i}\cdots\alpha_2\alpha_1(\hat{v}_x)$ for $i=1,2,\dots,k$. 
Let us prove that $w_i$ is nonzero for all $i=0,1,2,\dots,k$ by induction. The basis is obvious.
Assume $w_i\neq 0$. If $\alpha_{i+1}$ is adjacent to at least one vertex in $B$, we have
$\alpha_{i+1}(w_i)\neq 0$ directly by construction. Assume now that $\alpha_{i+1}$ is an arrow between 
two vertices in $A$. By construction, the only basis element in $V_{t(\alpha_{i+1})}$ which 
$\alpha_{i+1}$ annihilates is the one which is in the image of some path of length $N-2$.
We have $i<N-2$. Hence $\alpha_{i+1} w_i\neq 0$ if $t(\alpha_{j})\in A$ for all  $j\leq i$.
Otherwise let $j$ be maximal such that $j\leq i$ and $t(\alpha_j)\in B$. Then, by construction,
$\alpha_j(w_{j-1})$ is a non-zero multiple of $\hat{v}_{h(\alpha_j)}$, which implies that
$w_i$ is not in the image of a path of length $N-2$ and therefore $\alpha_{i+1}(w_i)\neq 0$ again.
This shows that $\omega$ acts in a nonzero way on $V$. As $\omega$ is a maximal path of length
strictly less than $N-1$, it is uniquely determined by the arrows it consists of. Injectivity of the
map $(\alpha,k)\mapsto p_{\alpha,k}$ thus implies that $\omega$ is uniquely determined by the
prime decomposition of the coefficients in its matrix. This completes the proof.
\end{proof}

Theorem~\ref{thm205} implies the following stabilization property for $\mathrm{eff.dim}_{\Bbbk}(\mathcal{P}_N)$:

\begin{corollary}\label{cor217}
Assume that $\Bbbk$ has characteristic zero. Then there exist  $a,b\in\mathbb{N}_0$ such that 
\[\mathrm{eff.dim}_{\Bbbk}(\mathcal{P}_N)=aN+b\quad\text{ for all }\quad N\geq n.\]
\end{corollary}

\begin{proof}
For each $x$ the numbers $l_x^-$ and $l_x^+$ satisfy $l_x^-+l_x^+\in\{0,1,\cdots,n-1,\infty\}$
as any path of length at least $n$ must contain a subcycle. This means that we always have one
of the following three cases:
\begin{itemize}
\item Both $l_x^-$ and $l_x^+$ are finite, and thus $l_x^-+l_x^+\leq n-1$. 
Then for all $N>n$ we have $l_x^-+l_x^++2-N\leq 1$ and $d_x=1$.
\item Exactly one of $l_x^-,l_x^+$ is finite. Then $d_x=\min\{l_x^-,l_x^+\}+1$ for all $N\geq n$.
\item Both $l_x^-,l_x^+$ are infinite. Then $d_x=N$ for all $N\geq 1$.
\end{itemize}
Therefore we can take $a$ to be the number of $x$ such that both $l_x^-,l_x^+$ are infinite.
As $b$ we take the sum of $1$'s over all $x$ such that both $l_x^-$ and $l_x^+$ are finite
plus the sum of $\min\{l_x^-,l_x^+\}+1$ over all $x$ such that exactly one of $l_x^-$ and $l_x^+$ is finite.
The claim follows.
\end{proof}

From Corollary~\ref{cor217} it follows that to calculate $\mathrm{eff.dim}_{\Bbbk}(\mathcal{P}_N)$
for all $N\in\mathbb{N}$ it is enough to consider the cases $N=1,2,\cdots,n,n+1$.

\section{Examples}\label{s4}

\subsection{Quivers with cycles at each vertex}\label{s4.1}
Let $Q$ be a quiver in which every vertex is part of some (nontrivial) cycle or loop. Then $\mathrm{eff.dim}_\Bbbk(\mathcal{P_N})=Nn$ for $\Bbbk$ uncountable or of characteristic 0. Proof: Let $x\in Q_0$. Then $l_x^-=l_x^+=\infty$ and hence $d_x=N$. Sum over all vertices. This result is similar to \cite[Theorem~31]{MS}, but the set of fields $\Bbbk$ differ.  

\subsection{Quivers of type $A_n$}\label{s4.2}
A quiver $Q$ is said to be of type $A_n$ if the underlying unoriented graph is the Dynkin diagram $A_n$. Let $Q$ be of type $A_n$ and let $(n_1,n_2,\cdots,n_k)$ be the the number of vertices in the ordered segments. Then
\[\eff(\mathcal{P})=1+\sum_{N<n_i}(N(n_i+1-N)-1) +\sum_{n_i\leq N}(n_i-1).\]
Proof: Because local dimensions $d_x$ only depend on maximal paths in and out of $x$, different ordered segments can be counted independently, if we subtract the overlaps. Thus we need only to consider the case when $Q$ has one ordered segment. For a quiver of type $A_n$ with only one ordered segment (with vertices from $\mathbf{1}$ to $\mathbf{n}$) the picture is as follows, when $N<n$. When $n \leq N$ each $V_x$ is one-dimensional.

$\xymatrix{\mathbf{1}\ar[r]&\cdots\ar[r]&\mathbf{N}\ar[r]%
&\cdots\ar[r]&\mathbf{n-(N-1)}\ar[r]&\cdots\ar[r]&\mathbf{n}\\
V_{\mathbf{1}}^{(0)}\ar[dr]&\cdots\ar[dr]&V_\mathbf{N}^{(0)}\ar[dr]&\cdots\ar[dr]&V_\mathbf{n-(N-1)}^{(0)}\ar[dr]\\
&\cdots\ar[dr]&\cdots\ar[dr]&\cdots\ar[dr]&\cdots\ar[dr]&\cdots\ar[dr]&\\
&&V_\mathbf{N}^{(N-1)}&\cdots&V_\mathbf{n-(N-1)}^{(N-1)}&\cdots&V_\mathbf{n}^{(N-1)}\\
}$

\noindent
Department of Math., Uppsala University,
Box 480, SE-751 06, Uppsala, Sweden; e-mail: {\tt love.forsberg\symbol{64}math.uu.se}
\vspace{2mm}

\end{document}